\documentclass{amsart}
\usepackage{hyperref}
\usepackage{esint}
\usepackage{amsmath,amsfonts,amsthm,amssymb}
\usepackage{mathtools}

\newtheorem{theorem}{Theorem}[]
\newtheorem{proposition}[theorem]{Proposition}
\newtheorem{lemma}[theorem]{Lemma}

\newtheorem{corollary}[theorem]{Corollary}

\theoremstyle{remark}

\begin{document}
\title[Arbitrarily Slow Decay]{Sets with Arbitrarily Slow Favard Length Decay}
\author{Bobby Wilson}
\date{\today}
\address{Department of Mathematics, Massachusetts Institute of Technology}
\email{blwilson@mit.edu}

\thanks{
The author would like to thank Alexander Volberg for introducing him to this problem and discussions thereafter.\\
This material is based upon work supported by the National Science Foundation under Grant No. DMS-1440140 while the author was in residence at the Mathematical Sciences Research Institute in Berkeley, California, during the Spring 2017 semester.}
\subjclass[2010]{28A75, 28A80}

\keywords{Differentiability, Favard Length, Cantor Set}

\begin{abstract}
In this article, we consider the concept of the decay of the Favard length of $\varepsilon$-neighborhoods of purely unrectifiable sets.  We construct non-self-similar Cantor sets for which the Favard length decays arbitrarily with respect to $\varepsilon$.
\end{abstract}
\maketitle

\section{Introduction}

Let $E \subset \mathbb{R}^2$ be a Borel set with $0< \mathcal{H}^1(E) <\infty$.  Let $\theta \in S^1$ represent a direction in $\mathbb{R}^2$ and $\ell_{\theta}$ be the line through the origin in direction $\theta$.  Finally, let $p_{\theta}: \mathbb{R}^2 \rightarrow \ell_{\theta}$ be defined as the orthogonal projection onto $\ell_{\theta}$.  For computational purposes, it may be convenient to define $p_{\theta}: \mathbb{R}^2 \rightarrow \mathbb{R}$ where $p_{\theta}(x):= x \cdot \theta$.  
We say that $E$ is rectifiable (or countably rectifiable) if there exists a countable collection of Lipschitz maps, $f_i: [0,1] \rightarrow \mathbb{R}^2$, such that 
	\begin{align*}
		\mathcal{H}^1\Big( E \setminus \big[\cup_k f_k([0,1])\big] \Big)=0
	\end{align*}
	Furthermore, we say that $E$ is purely unrectifiable if, for every Lipschitz function, $f:[0,1] \rightarrow \mathbb{R}^2$,
	\begin{align*}
		\mathcal{H}^1\Big( E \cap f([0,1]) \Big)=0.
	\end{align*}
We can define rectifiability for higher dimensional sets, in which case the sets defined above are known as 1-recitifiable and 1-purely unrectifiable.  We will concern ourselves with only one-dimensional sets in this article, so we suppress references to dimension.	

Next, we define a notion known as the Favard length of a set $E$:
	\begin{align*}
		\mbox{Fav}(E) := \fint_{S^1} |p_{\theta}(E) | \, d\theta, 
	\end{align*}
	where $\fint$ denotes the average, $\fint_E f \,d\mu=\mu(E)^{-1}\int_E f\,d\mu$. One part of  Besicovitch's projection theorem, \cite{besi3}, asserts that if $E$ is purely unrectifiable, then $\mathcal{H}^1( p_{\theta}(E) )=0$ for almost every $\theta \in S^1$.  Thus, if a set $E$ is purely unrectifiable, then $\mbox{Fav}(E)=0$. 

 For any $\varepsilon>0$, define the $\varepsilon$-neighborhood of a set $E$ by
	\begin{align*}
		\mathcal{N}(E, \varepsilon):= \{x \in \mathbb{R}^2~|~ \mbox{dist}( x, E)\leq\varepsilon  \}
	\end{align*}

The question we would like to consider is: For which function $\phi:(0,1] \rightarrow (0,\infty)$ do we have 
	\begin{align*}
	0< c_E\leq \frac{\mbox{Fav}(\mathcal{N}(E, \varepsilon))}{\phi(\varepsilon)} \leq C_E<\infty?
	\end{align*}

For a compact, purely unrectifiable set $E$, $\lim_{\varepsilon \rightarrow 0}\mbox{Fav}(\mathcal{N}(E, \varepsilon))=\mbox{Fav}(\mathcal{N}(E, 0))$.  This follows from the monotonicity of Favard length.  In particular, monotonicity implies that for a sequence of compact, nested sets $E_1 \supset E_2 \supset \cdots$,
	\begin{align*}
		\lim_{\varepsilon \to 0}\mbox{Fav}(E_n) = \mbox{Fav}\left( \bigcap_{n=1}^{\infty} E_n \right).
	\end{align*} 
	
Most results regarding Favard length consider different types of Cantor sets.  Denote the classical Four-Corner Cantor set by $\mathcal{C}_4 \subset \mathbb{R}^2$.  Generally, we can define a Cantor set via a family, $\{ T_1, ..., T_L\}$, $T_j: \mathbb{C} \rightarrow \mathbb{C}$ of similarity maps of the form $T_j(z) = \frac{1}{L}z + z_j$, where $z_1, ..., z_L$ are distinct and not co-linear.  We then let $S_{\infty}$ be defined as the unique compact set such that $S_{\infty} = \cup_{j=1}^L T_j(S_{\infty})$. 

The first estimate of the decay of Favard length is due to Mattila.  In \cite{mat90}, he provides a logarithmic lower bound for the Cantor set: $\mbox{Fav}(\mathcal{N}(S_{\infty}, \varepsilon)) \geq C  [\log \varepsilon^{-1}]^{-1}$.  The first upper bound for the generalized Cantor set is an iterated logarithm bound
			\begin{align*}
				\mbox{Fav}(\mathcal{N}(S_{\infty}, \varepsilon)) \leq C \exp( C \log^* \log \varepsilon^{-1}) 
			\end{align*}
			(where $\log^* N= \min \{ n\geq 0~|~ \log \log \cdots \log N \leq 1 \}$) provided by Peres and Solomyak \cite{PS02}.  They also provide a beautiful proof showing that Mattila's estimate for a randomly constructed Four-Corner Cantor set is the best possible bound. In other words, they prove the logarithmic upper bound
			\begin{align*}
				\mathbb{E}_{\omega} \left[\mbox{Fav}(\mathcal{N}(\mathcal{C}^{\omega}_4, \varepsilon)) \right]\leq \frac{C}{ \log \varepsilon^{-1}} 
			\end{align*}
A surprising fact is that although Mattila's lower bound is almost surely the rate at which random Four-Corner Cantor sets decay, it is not the best lower bound for $\mathcal{C}_4$. Bateman and Volberg \cite{BaV10} prove a better lower bound for the Four-Corner Cantor set:
			\begin{align*}
				\mbox{Fav}(\mathcal{N}(\mathcal{C}_4, 4^{-n})) \geq C \frac{\log \log 4^n}{\log 4^n}\geq C'\frac{\log n}{n}.
			\end{align*}
Nazarov, Peres, and Volberg \cite{NPV11} provide what appears to be the type of result we should expect to be valid, in some sense, for all self-similar Cantor sets.  In particular,
		\begin{align*}
			\mbox{Fav}(\mathcal{N}(\mathcal{C}_4, 4^{-n})) \leq Cn^{-p+\delta}
		\end{align*}
for $p=\frac{1}{6}$ and any $\delta>0$.

Following these estimates, many results were proven for particular types of Cantor sets.  Bond and Volberg, \cite{BV10}, prove $\mbox{Fav}(\mathcal{N}(S_{\infty}, \varepsilon)) \leq C  [\log \varepsilon^{-1}]^{-p}$ for some $p>0$ when $L=3$. Furthermore, Bond, \L aba, and Volberg, \cite{BLV14}, prove $\mbox{Fav}(\mathcal{N}(S_{\infty}, \varepsilon)) \leq C [\log \varepsilon^{-1}]^{-p}$ for some $p>0$ when $L=4$. Finally, \L aba and Zhai \cite{LZ} prove that 
			\begin{align*}
				\mbox{Fav}(\mathcal{N}(S_{\infty}, \varepsilon)) \leq C  (\log \varepsilon^{-1})^{-p} 
			\end{align*}
for product Cantor sets with a ``tiling condition".  A more thorough review of the work produced about this subject can be found in a survey of \L aba \cite{L15}.  A slightly different type of result was proven by Bond, \L aba and Zahl \cite{BLZ16} relating the Favard length of self-similar Cantor sets to the visibility.

In many of these results, self-similarity plays a crucial role in establishing estimates on the Favard length. This suggests that the logarithmic bounds are linked, in part, to the self-similarity of the constructions. Through what is proven in this article, we hope to suggest that the estimates demonstrated before are tied to the Hausdorff measure density of the Cantor set construction.
We would like to address a question not answered with Cantor sets defined via iterated function systems.  Particularly, constructing purely unrectifiable sets for which the Favard length decays at any rate  chosen.

\begin{theorem}\label{mainthm}
Let $g: \mathbb{N} \rightarrow \mathbb{R}_+$ be a monotonic sequence of positive numbers such that $\lim_{n \rightarrow \infty} g(n)= \infty$.  Then there exists a measurable set $E$ such that $\mathcal{H}^1(E)=1$, $E$ is purely unrectifiable, and
	\begin{align*}
		\mbox{Fav}(\mathcal{N}(E, 4^{-n})) \gtrsim \frac{1}{g(n)}.
	\end{align*}
\end{theorem}

The idea of this theorem was inspired by the construction of a counterexample to Besicovitch's projection theorem in infinite-dimensional Banach spaces appearing in the paper by Bate, Cs\"ornyei and Wilson \cite{BCW}.  In \cite{BCW}, for any separable Banach space, $X$, the authors construct a one-dimensional purely unrectifiable set for which each projection has positive measure. In this paper, we have a similar goal in constructing purely unrectifiable sets with with big projections, so the construction presented here follows very closely to the construction in \cite{BCW}.

The following statement is an obvious corollary to Theorem \ref{mainthm} that allows for the replacement of the discrete sequence with a continuous function.  
		\begin{corollary} \label{maincor}
			Let $\phi: (0,1] \rightarrow \mathbb{R}_+$ be a monotonic function  defined so that 
	\begin{align*}
		\lim_{\varepsilon \rightarrow 0} \phi(\varepsilon)= 0.
	\end{align*}
  Then there exists a measurable set $E$ such that $\mathcal{H}^1(E)=1$, $E$ is purely unrectifiable, and
	\begin{align*}
		\mbox{Fav}(\mathcal{N}(E, \varepsilon)) \gtrsim \phi(\varepsilon).
	\end{align*}
		\end{corollary}

There are results that, one could argue, suggest that sets with arbitrarily slowly decaying Favard length need not exist.  Particularly, Marstrand's density theorem tells us that if a set $E \subset \mathbb{R}^2$, $\mathcal{H}^1(E)<\infty$, is purely unrectifiable, then
	\begin{align*}
		\Theta^1_*(E,x):=\liminf_{r \rightarrow 0} \frac{ \mathcal{H}^1(E \cap B(x, r))}{2r} <1
	\end{align*} 
for $\mathcal{H}^1$ almost every $x \in E$.  Furthermore, if $E$ is rectifiable, then $\Theta_*^1(E,x)=1$ for almost every $x \in E$. Of course, a classical result of Besicovitch \cite{mattila} showed that there are no sets $E \subset \mathbb{R}^2$ such that $\Theta_*^1(E,x)\in (\tfrac{3}{4}, 1)$ for almost every $x \in E$.  Thus, in some sense, the behavior of purely unrectifiable sets can not be arbitrarily similar to that of rectifiable sets.  One could initially think that the difference between the two cases is perhaps be due to the fact that the value of $\Theta^1_*(E,x)$ is more closely tied to the Euclidean structure of $\mathbb{R}^2$ than $\mbox{Fav}(E, \varepsilon)$ .  However, Preiss and Ti\v ser \cite{PT92}, show that this density gap exists, and is larger than $\tfrac{1}{4}$, for all metric spaces.  Therefore, from this point of view it could come as a surprise that in no sense does a gap exist for this characterization of rectifiability.

This paper is organized as follows:  first, we will detail the construction of the set.  In the following section, we will state some preliminary definitions and lemmas.  We will then compute the Favard length, and finish with a section showing that these sets are purely unrectifiable.

\section{Construction}

The construction of this set is similar to the construction of a four-corner Cantor set.  The difference will be that at every step of the construction the squares will be separated at a smaller distance than the previous step.  One should think of this construction as approaching the trivial construction of a straight line segment via an iterated construction method.

 Define
	\begin{align*}
		f(x):= \left\{ \begin{array}{ll}  0 & x\in  [0,\tfrac{1}{2})+\mathbb{Z},  \\
						3/4 & x\in [\tfrac{1}{2}, 1)+\mathbb{Z} \end{array}\right. 
	\end{align*}

Fix a sequence $(g(k))_{k=1}^{\infty}$ such that $g(k) \rightarrow \infty$ and let
	\begin{align*}
		a_1 &:= g(1)\\
		a_k &:=\mbox{min}[1, g(k)-g(k-1)] \hspace{1cm} k\geq 2.
	\end{align*}
For sequences $g$ that grow super-linearly, $\mathcal{C}_4$ is a suitable example that can be used to prove Theorem \ref{mainthm}.  So it suffices to consider $g$ that grow linearly or sublinearly.  These provide the most interesting examples in this paper.  Now we define a sequence of positive integers $(m_k)_{k=1}^{\infty}$ such that $m_k>m_{k-1}$, $1000 \cdot 4^{-m_{k}} \leq a_{k}4^{-m_{k-1}}$, and $1000 \cdot 4^{-m_{k}} \leq a_{k-1}4^{-m_{k-1}}$.

Next, let $s_k$ be the vertical line segment, $s_k:= \{0\}\times [0,4^{-m_k}]$. For each $n \in \{1, 2,...,\infty\}$, let
	\begin{align*}
		&f_n: [0,1] \rightarrow \mathbb{R}\\
		&f_n(x) := f(2x)+\sum_{j=1}^n a_j4^{-m_j} f(2\cdot4^{m_j}x), \hspace{1cm} x \in [0,1]
	\end{align*}
and
	\begin{align*}
		E_n:= \overline{s_n + \mbox{graph } f_n}.
	\end{align*}
Figure 1 demonstrates an example of a construction of $E_n$ for $n=3$.

Furthermore, let $\mathcal{F}_n := \mbox{graph } f_n$.  We note that 
	\begin{align*}
		\mathcal{F}_{\infty} := \lim_{n \rightarrow \infty} \mathcal{F}_n = \lim_{n \rightarrow \infty} E_n =: E_{\infty}
	\end{align*}

The use of powers of four in the constructions are used to conform the structures of $E_{\infty}$ and $\mathcal{F}_{\infty}$ to that of the four corner Cantor set, and are purely for cosmetic purposes. Four can be replaced by any positive integer greater than one.
$$
\begin{array}{c} \mbox{Figure 1}\\
\includegraphics[scale=0.33]{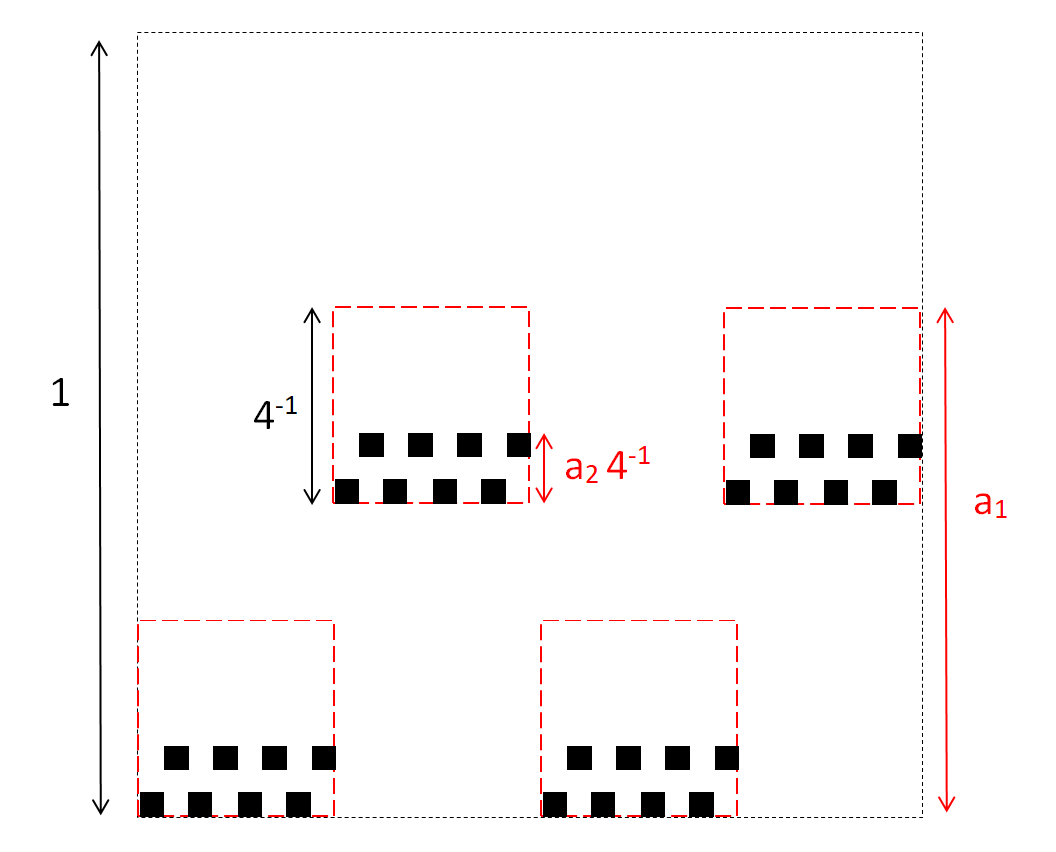}
\end{array}
$$

\section{Preliminaries}

For $A, B \in \mathbb{R}_+$, when we say that $A \lesssim B$ we mean that there exists a $C>0$ such that $A \leq C \cdot B$. Furthermore, when $A \sim B$, there exists $C>0$ such that $C^{-1} \cdot A \leq B \leq C \cdot A$.   $\mathcal{L}^k$ represents the $k$-dimensional Lebesgue measure while $\mathcal{H}^k$ represents the $k$-dimensional Hausdorff measure.  
 The Hausdorff measure $\mathcal{H}^k$ is defined as the limit, $\lim_{\delta\to 0} \mathcal{H}_\delta^k$, where for a given set $E$, 
\begin{align*}
\mathcal{H}_\delta^k(E)=\inf\left\{\sum_i (\mbox{diam}(E_i))^k~\big|~E\subset\bigcup_i E_i,\  \mbox{diam}(E_i)<\delta\right\}.
 \end{align*}  

For each $n \in \mathbb{N}$, let $\mathcal{I}_n$ denote the collection of $4^{-m_n}$-length intervals of the form $[4^{-m_n}k, 4^{-m_n}(k+1))$, for $k \in \mathbb{N}$, in $[0,1]$. The following application of the Borel-Cantelli lemma will be essential,

\begin{lemma}\label{2.3}
Let  $(a_n)\not\in\ell_1$ be a sequence of positive real numbers and let $(m_n)_{n=1}^{\infty}$ be a sequence of positive integers satisfying 
	\begin{itemize}
		\item $m_n>m_{n-1}$ and 
		\item $1000 \cdot 4^{-m_n} \leq a_n 4^{-m_{n-1}}$.
	\end{itemize}
 Then for $\mathcal{L}^1$ almost every $t\in [0,1]$ there are infinitely many $n$ such that $d(t,4^{-m_n}\mathbb N)\leq 4^{-m_n}a_n.$
\end{lemma}

\begin{proof}
For each $n$, let $A_n$ be the integer part of $a_n4^{m_n-m_{n-1}}$ and let $b_n:= A_n4^{m_{n-1}-m_n}$.  Then it is clear, since $1000 \cdot 4^{-m_{k}} \leq a_{k}4^{-m_{k-1}}$, that $\frac{1}{2}a_n \leq b_n \leq a_n$ for each $n$. 

Each $I\in\mathcal{I}_{n-1}$ has $4^{m_n-m_{n-1}}$ subintervals in $\mathcal{I}_{n}$, and $d(t,4^{-m_n}\mathbb N)\leq 2 \cdot 4^{-m_n}a_n$ holds for at least $b_n 4^{m_n-m_{n-1}}$ many subintervals.
These events are independent and a direct application of the Borel-Cantelli lemma completes the proof.
\end{proof}

Next, we show that $f_{\infty}$ satisfies a Luzin condition that is, in a certain way, a measure theoretic Lipschitz condition for the graph of a function.

\begin{lemma}\label{2.2} For any measurable set $S\subset [0,1]$,
	\begin{align*}
		\mathcal{H}^1(\{(x, f_{\infty}(x)) : x\in S\}) \leq K\, |S|.
	\end{align*}
In particular, $\mathcal{H}^1(E)<2$, and $f$ satisfies Luzin's condition:
	\begin{align*}
		\mathcal{H}^1(\{(x, f_{\infty}(x)) : x\in N\})=0
	\end{align*} 
for any Lebesgue null set $N\subset[0,1]$. 
\end{lemma}
It will become apparent in Section \ref{ret} that it is crucial that $K<2$.  In fact, we can take $K = 1$.
\begin{proof}
Let $I=I_{m,n} \in \mathcal{I}_n$ denote the dyadic interval $[m4^{-m_n},(m+1)4^{-m_n})$ for $m,n\in\mathbb{N}$. 
For $k>n$ the function $f_k$ is piece wise constant and thus
	\begin{align*}
		\mathcal{H}^1(\{(x, f_{k}(x)) : x\in I \}) = \sum_{J \subset I \atop J \in \mathcal{I}_k} \mathcal{H}^1(\{(x, f_{k}(x)) : x\in J \})=|I|
	\end{align*}
On $J$, the function $f_{\infty}$ oscillates at most $\frac{1}{2}4^{-m_k}$. Therefore, 
	\begin{align*}
		s_k+\mbox{graph } f_k|_I \supset \{(x, f_{\infty}(x)) : x\in I \}
	\end{align*}
which implies that $\mathcal{H}^1_{4^{-m_k}}(\{(x, f_{\infty}(x)) : x\in I \}) \leq |I|$ for each $k>n$.  Thus, $\mathcal{H}^1(\{(x, f_{\infty}(x)) : x\in I \}) \leq |I|$

The statement for a general measurable set $S$ follows by approximating $S$ by a countable union of dyadic intervals.
\end{proof}

\section{Favard Length}

We will now estimate the Favard length of $E_n$.  We will use computation techniques seen in the paper of Bateman and Volberg \cite{BaV10} as well as a technique for the construction of a Kakeya set which uses point-line duality to provide an easy way to relate the Favard length of simple geometric objects to the Lebesgue area of the corresponding dual set.  It will be much simpler to compute that Favard length of $\mathcal{F}_n$ in place of $E_n$, and since $\mathcal{F}_n \subset E_n$ implies $\mbox{Fav}(\mathcal{F}_n) \leq \mbox{Fav}(E_n)$, it suffices to do so for our purposes.
	\begin{proposition}\label{mainprop}
		\begin{align*}
			\mbox{Fav}(E_n) \gtrsim \left(\sum_{k=1}^{n-1} a_k\right)^{-1}.
		\end{align*}
	\end{proposition}

	For a fixed $n$, let $\{S\}$ denote the collection of line segments that compose $\mathcal{F}_n$.  For any angle $\theta$, and segment $S$, let $\chi_{S, \theta}(x)$ be the indicator function for the projection of $S$ onto the line centered at the origin with angle $\theta$. Furthermore, let $\mathcal{F}_{n,\theta}$ be the projection of $\mathcal{F}_n$ to the  the line centered at the origin with angle $\theta$.   The remainder of this section is devoted to the proof of Proposition \ref{mainprop}.
%

	
	\subsection{Point-Line Duality}
		We will compute the Favard length using point-line duality.  Consider the following correspondence between points in $\mathbb{R}^2$ and lines in $G(2,1)$:
			\begin{align*}
				P_{\ell}=(a,b)\in \mathbb{R}^2 \leftrightarrow \ell_P =\{ (x, y) \in \mathbb{R} ~|~ y=ax+b\} \in G(2,1)
			\end{align*}

		For a set $E \in \mathbb{R}^2$, let $E^*$ denote the set defined as
			\begin{align*}
				E^*:=\bigcup_{P \in E} \ell_P
			\end{align*}
		The important geometric property that this correspondence induces is contained in the following simple lemma that gives a similar characterization of Favard length. 
		The lemma requires a follows from the observation that, with the correspondence $\theta=(\theta_1, \theta_2) \leftrightarrow (\tfrac{\theta_1}{\theta_2}, 1)$, we can say
			\begin{align*}
				\theta^{-1}_2 \cdot p_{\theta}(E) &= \left\{(\tfrac{\theta_1}{\theta_2}, 1)\cdot (a,b) ~|~   (a,b) \in E\right\}\\
					&=\bigcup_{(a,b) \in E} \left\{ax+b~|~ x=\tfrac{\theta_1}{\theta_2}\right\}= \left\{x=\tfrac{\theta_1}{\theta_2}\right\} \cap E^*
			\end{align*}
where, for $a\in \mathbb{R}$, and $F \subset \mathbb{R}$, $a \cdot F:= \{a \cdot f ~|~ f \in F\} $. Thus heuristically we can think of $\mbox{Fav}(E)$ as being similar to $\mathcal{L}^2 (E^* )$.  Having to normalize the projections makes this an imprecise statement.  Particularly when $\theta_1/\theta_2\gg1$. We can avoid this technicality because we are trying to establish a lower bound.

		\begin{lemma}
			Let $E \subset \mathbb{R}^2$ be a measurable set satisfying $\mathcal{H}^1(E)<\infty$.  Then 
				\begin{align*}
					\mbox{Fav} (E) \gtrsim \mathcal{L}^2(E^* \cap Q) 
				\end{align*}
			where $Q$ is the vertical strip $Q=[0,1] \times\mathbb{R}$.
		\end{lemma}

		\begin{proof}
			This proof relies on the fact that if $\tfrac{\theta_1}{\theta_2} \in [0,1]$ then $\theta_2 \in [\sqrt{2}^{-1}, 1]$.  Acknowledging this, let $G := \{\theta \in S^1~|~ \tfrac{\theta_1}{\theta_2} \in [0,1]$ we have
				\begin{align*}
					\mbox{Fav}(E) &= \int_{S^1} |p_{\theta}(E)| \,d\theta \geq \int_{G} |p_{\theta}(E)|\,d\theta\\
								&\gtrsim \int_{G} \theta_2 |p_{\theta}(E)|\,d\theta =  \int_{G}  |\theta_2\cdot p_{\theta}(E)|\,d\theta\\
								&\gtrsim \int_0^1 |\{x=\xi\} \cap E^*\}| \,d\xi\\
								&= \mathcal{L}^2(E^* \cap Q).
				\end{align*}
		\end{proof}

		The crucial benefits of this lemma are two-fold.  First, computing the area of $E^*$ is much simpler than computing $\mbox{Fav}(E)$.  Second, the closer $S_1$ and $S_2$ are to each other on the $y$-axis the more likely that $(S^*_1 \cap S^*_2) \bigcap Q = \emptyset$.  In fact, if $(S^*_1 \cap S^*_2) \bigcap Q \neq \emptyset$ and $|p_{(0,1)}S_1-p_{(0,1)}S_2| \leq \delta$, then for some $C>0$,
		\begin{align*}
			\mbox{dist}(p_{(1,0)}S_1, p_{(1,0)}S_2) \leq C\delta.
		\end{align*}
		This will be of importance when demonstrating the final inequality of the following estimate.  Using Cauchy-Schwartz, as in \cite{BaV10}, we have

\begin{align*}
			1 =4^{m_n} \cdot 4^{-m_n} &\sim  \int_Q \sum_{S}\chi_{S^*}(x) \,dx\\
				& \leq \left( \mathcal{L}^2\big((\bigcup_{S \subset \mathcal{F}_n} S^*) \cap Q\big) \right)^{1/2} \left(  \int_Q \left[\sum_{S}\chi_{S^*}(x)\right]^2 \,dx\right)^{1/2}.
		\end{align*} 
Of course, $\bigcup_{S \subset \mathcal{F}_n} S^*= \mathcal{F}_n^*$ and thus
		\begin{align*}
			\left(\int_{Q} \left( \sum_{ S \subset \mathcal{F}_n}  \chi_{S^*} \right)^2 \right)^{-1} \lesssim \mathcal{L}^2(\mathcal{F}^*_n \cap Q).
		\end{align*}
We have now reduced our estimate on Favard length to an estimate on pairwise intersection of strips.
		\begin{align}
			\left(\sum_{S_1, S_2} \mathcal{L}^2(S^*_1 \cap S^*_2 \cap Q) \right)^{-1}&=\left(\int_{Q}  \sum_{ S_1, S_2 \subset \mathcal{F}_n}  \chi_{S_1^*}\chi_{S_2^*}  \right)^{-1} \nonumber\\
&= \left(\int_{Q} \left( \sum_{ S \subset \mathcal{F}_n}  \chi_{S^*} \right)^2 \right)^{-1} \lesssim \mbox{Fav}(\mathcal{F}_n)\nonumber \\
& \lesssim \mbox{Fav}(E_n)
		\end{align}
Figure 2 illustrates the intersection of a pair of strips that are dual to segments in $\mathcal{F}_n$.  In the next subsection, we detail the geometric and combinatorial computations that will conclude the proof of Proposition \ref{mainprop}.
$$
\begin{array}{c} \mbox{Figure 2}\\
\includegraphics[scale=0.3]{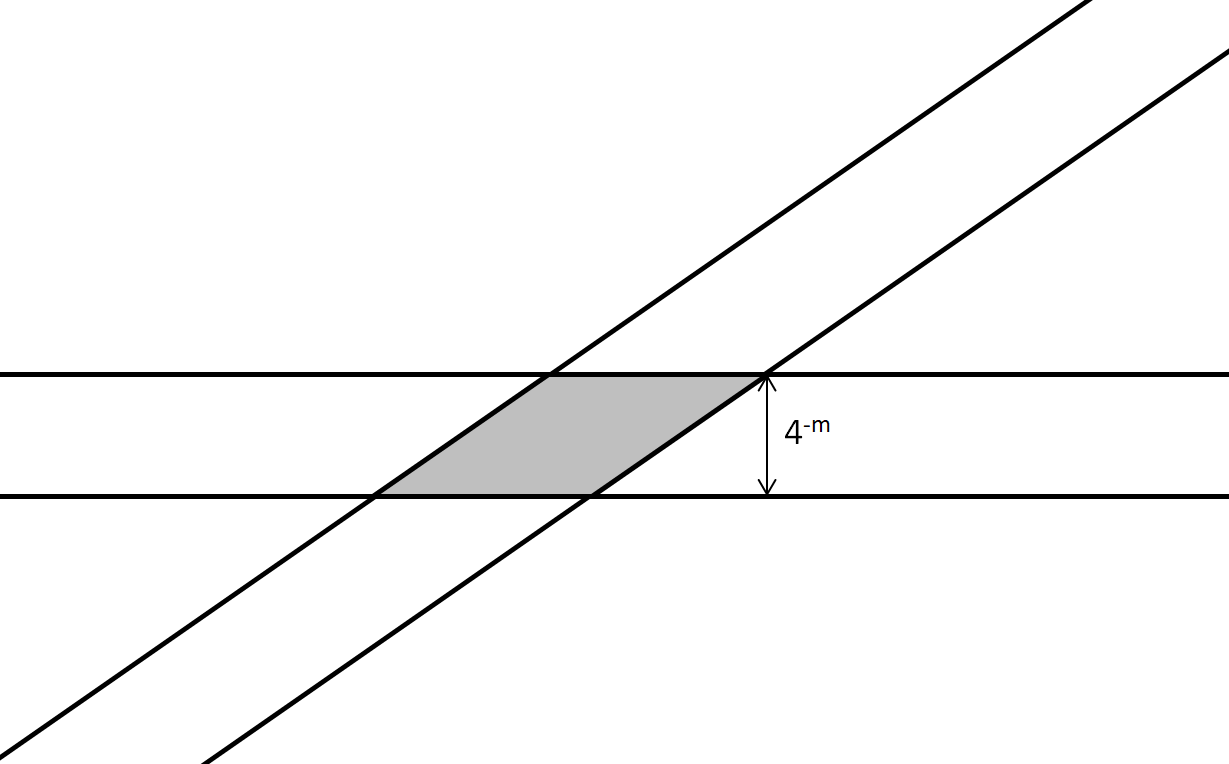}\\
S^*_1 \cap S^*_2
\end{array}
$$

	\subsection{Proof of Proposition \ref{mainprop}}
		Again, borrowing from Bateman and Volberg, we call a pair of $4^{-m_n}$ length segments a $k$-pair if $S_1, S_2$ are in one $ 4^{-m_k}$
 square, but not in any $4^{-m_{k+1}}$ square, for all $k \in \{ 1, ..., n\}$. We further require that $S_1$ and $S_2$ are no further than $Ca_k4^{-m_k}$ away from each other horizontally for some fixed constant $C>0$. We can enforce this requirement by the argument from above. Particularly, if $S_1$ and $S_2$ are in one $4^{-m_k}$ square then $|p_{(0,1)}S_1 - p_{(0,1)}S_2| \leq 10 \cdot a_k4^{-m_k}$ and so $S_1^* \cap S^*_2 \cap Q \neq \emptyset $ if and only if  $S_1$ and $S_2$ are no further than $Ca_k4^{-m_k}$ away from each other horizontally.  By a simple calculation (that can be found in \cite{BaV10}) we have  $\sim 4^{m_k} \cdot ( a_k4^{m_n-m_k})^2$,  $k$-pairs for each $k \in \{1,...,n-1\}$.  For $k=n$, $S_1=S_2$.

		By a simple geometric computation,

			\begin{align*}
				\mathcal{L}^2(S^*_1 \cap S^*_2 \cap Q) \lesssim \frac{4^{-2m_n}}{4^{-m_k}a_k}.
			\end{align*}

		Thus,
			\begin{align*}
				 \sum_{S_1, S_2} \mathcal{L}^2(S^*_1 \cap S^*_2 \cap Q)
				&\lesssim \sum_{k=1}^n \sum_{S_1, S_2 \, k-\mbox{pair}} \mathcal{L}^2(S^*_1 \cap S^*_2 \cap Q)\\
				&\lesssim \sum_{k=1}^n \sum_{S_1, S_2 \, k-\mbox{pair}} \frac{4^{-2m_n}}{4^{-m_k}a_k}\\
				&\lesssim 1+\sum_{k=1}^{n-1} a_k
			\end{align*}
		This is exactly what we want to show.

\section{Rectifiability} \label{ret}
	\begin{proposition}\label{rect}
		$E_{\infty}$ is purely unrectifiable.
	\end{proposition}

Of course this implies that 
	\begin{align*}
		\lim_{n \rightarrow \infty} \mbox{Fav}(E_n)= \mbox{Fav}(E_{\infty}) =0
	\end{align*}

We will use the following lemma that appears in the paper of Bate, Cs\"ornyei and Wilson \cite{BCW} as a reformulation of a lemma from the work of Kirchheim, \cite{kirchheim}.  This lemma is far more broad than we need for this application.  A simpler corollary may be available, but is not presented here. 

\begin{lemma}\label{l:kirchheim}
Let $X$ be a metric space.  Suppose that $E\subset X$ and $\gamma : [0,1]\to X$ is Lipschitz with
	\begin{align*}
		\mathcal{H}^1(\gamma([0,1]) \cap E)>0.
	\end{align*}
Then there exist a measurable $A\subset [0,1]$ of positive measure and an $L \geq 1$ such that $\gamma(A):= F \subset E$, $\gamma$ restricted to $A$ is bi-Lipschitz with bi-Lipschitz constant $L$, and such that for $\mathcal{H}^1$-a.e. $y_0 \in F$ and every $\varepsilon>0$, if $r$ is sufficiently small then 
	\begin{align}\label{density}
		\mathcal{H}^1(F\cap B(y_0,r))\geq 2r(1-\varepsilon).
	\end{align}
\end{lemma}

The proof of Proposition \ref{rect} will follow closely to the proof of pure unrectifiability in \cite{BCW} (Proposition 5). 

	\begin{proof}[Proof of Proposition \ref{rect}]
		 It suffices to prove that $\mathcal{F}_{\infty}$ is purely unrectifiable since $ \mathcal{F}_{\infty}=E_{\infty}$. 

		Suppose that $\mathcal{F}_{\infty}$ is not purely unrectifiable.  Let $\gamma$, $F$, and $A$ be defined as in the statement of Lemma \ref{l:kirchheim}.  By Lemma \ref{2.3}, for almost every $x_0 \in [0,1]$, there are infinitely many $n$ for which $\mbox{dist}(x_0, 4^{-m_n}\mathbb{N})< a_n4^{-m_n}$.  The Luzin condition (Lemma \ref{2.2}) then implies that for $\mathcal{H}^1$-a.e. $y_0=(x_0, f_{\infty}(x_0)) \in F$, 
		\begin{align}\label{close}
			\mbox{dist}(x_0, 4^{-m_n}\mathbb{N})< a_n4^{-m_n}.
		\end{align}

		Fix $y_0= \gamma(s_0)= (x_0,f_{\infty}(x_0))$ for which \eqref{density} holds for every $\varepsilon>0$ and sufficiently small $r>0$ and for which \eqref{close} holds for infinitely many $n$.  Note that we can can choose $s_0$ uniquely since $\gamma$ is a bijection and we can assume that $s_0$ is a density point of $A$.

		By Lemma \ref{2.3} and \eqref{density}, for small enough $r$,
		\begin{align}
			\mathcal{H}^1\left[ f_{\infty}^{-1}(F \cap B(y_0, r)) \right] \geq \tfrac{1}{K} \mathcal{H}^1\left[ F \cap B(y_0, r) \right]\geq \tfrac{1}{K} 2r(1-\varepsilon)
		\end{align}
		We will choose $\varepsilon>0$ small enough so that $\tfrac{1}{K} 2r(1-\varepsilon)>r(1+\delta)$. In this we can, for each $n$ for which $x_0$ satisfies \eqref{close}, choose $z_1, z_2 \in f^{-1}_{\infty}(F \cap B(y_0,a_n4^{-m_n}))$ so that $z_1$ belongs to the same $4^{-m_n}$-length interval and $z_2$ belongs to an adjacent $4^{-m_n}$-length interval.  Thus, 
		\begin{align*}
			|f_n(z_1)-f_n(x_0)| =0 \mbox{ and } |f_n(z_2)-f_n(x_0)|=a_n4^{-m_n}
		\end{align*} 
which, along with $|f_n(t)- f_{\infty}(t)|<\frac{1}{100}a_n4^{-m_n}$ for $t \in [0,1]$, implies 
		\begin{align}\label{disc}
			&|f_{\infty}(z_1)-f_{\infty}(x_0)| \leq \tfrac{1}{10}a_n4^{-m_n} \mbox{ and }\\  \label{disc2}&|f_{\infty}(z_2)-f_{\infty}(x_0)|\geq \tfrac{9}{10}a_n4^{-m_n}.
		\end{align} 
For $i=1, 2$, let $\ell^n_i$ be the line through $y_0$ and $(z_i, f_{\infty}(z_i))$.  Inequalities \ref{disc} and \ref{disc2} imply that the angles between the two secant lines stay a fixed distance apart for each $n$.  In other words, for each $n$ satisfying \eqref{close}, 
		\begin{align*}
			\angle ( \ell^n_1, \ell^n_2)\gtrsim 1
		\end{align*}
 where the constant does not depend on $n$.  Finally, this contradicts the assumption that $y_0$ is a tangent point for $\gamma$.
	\end{proof}

\section{Conclusion}

Theorem \ref{mainthm} and Corollary \ref{maincor} suggest that if there is a connection between the lower density of a set and the behavior of the Favard length of the $\varepsilon$-neighborhood of the set it may not be as simple as a functional correspondence.  For any well-ordering of decreasing functions, and corresponding metric, we can always find a set whose Favard length decreases slower than a function arbitrarily close to the constant function in that metric.  We reiterate that this is very different from the connection between lower Hausdorff density and rectifiability.  

One topic that has not been discussed thus far in this article is the opposite question to the one posed here.  The question of Favard length decreasing at a arbitrarily fast speed.  Of course, for measure zero sets, we can take the single point set and see that $\mbox{Fav}(E, \varepsilon) \leq \varepsilon$ and thus linear decay is the fastest decay.  However,  a singleton set is a zero dimensional set.  An interesting question to ask  would be how fast a positive measure one-dimensional set's Favard length can decay.

\end{document}